\documentclass[11pt]{article}
\usepackage{amsfonts} 
\usepackage{amsmath}
\usepackage{amsthm}
\usepackage{graphicx}
\usepackage{lineno}
\usepackage{hyperref}
\modulolinenumbers[5]

\newtheorem{theorem}{Theorem}[section]

\newtheorem{lemma}[theorem]{Lemma}
\newtheorem{case}{Case}

\begin{document}

\title{The 4-move kills the Alexander polynomial\thanks{I started this paper in October of 2020 as a response to the Covid-19 lockdown imposed at the time. I worked on it for about a month back then, as a distraction from working in a lockdown. I picked it up again in October 2025 in my last month at the Institute of Labor Economics. As I realize five years later the proof was practically finished back in 2020.}}
\author{Nikos Askitas\\Bonn}
\date{}

\maketitle

\begin{abstract}
Whether or not the 4-move is an unknotting operation remains an open problem. In this paper I show that every knot can be reduced to one with a trivial Alexander polynomial via a sequence of $4$-moves and isotopies.
\end{abstract}


\bigskip

\nolinenumbers

\section{Introduction}

The $4$-move is defined as in Fig. \ref{fig:4move} below. It is an unsettled, as of now, conjecture of Yasutaka Nakanishi since 1979 that the $4$-move is an unknotting operation. 
\begin{figure}[!htb]
	\centering
  \includegraphics[width=.3\textwidth]{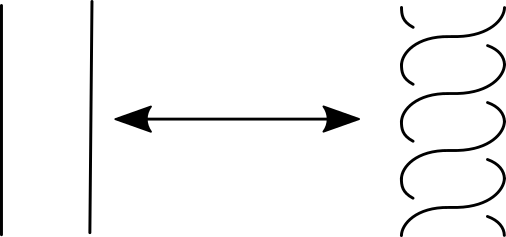}
	\caption{The 4-move }
	\label{fig:4move}
\end{figure}
An excellent account of the complete story so far can be found in \cite{joseph}. In short the conjecture, posted in 1994 by J. Przytycki as Conjecture 2.2 (a) in Rob Kirby's list of Problems in Low-Dimensional Topology, is known to hold for many knots including knots of up to 12 crossings. For many years the $2$-cable of the trefoil was the smallest potential counterexample to the conjecture. I managed, however, to unknot it in 1999 in \cite{nikos1}. Answering a question by J. Przytycki I then drew the tangle analogue of the unknotting movie in \cite{nikos2} and proposed the $2$-cable of the figure eight knot as the smallest potential counterexample which still stands open. In this paper building on \cite{nikos1, nikos2, effie} I prove that every knot can be reduced to a knot with trivial Alexander polynomial using $4$-moves:
\begin{theorem} \label{theorem:main}
For every knot $K$ there is a knot $K^{'}$ obtained from $K$ by a finite sequence of $4$-moves and isotopies so that the Alexander polynomial $\Delta(K^{'})$ is trivial.
\end{theorem}
The proof is based on the fact that knot projections can be constructed out of certain trivalent graph projections which were used in \cite{effie}. In Section \ref{section:geometric} I prove several elementary geometric lemmas about the effect of $4$-moves on such knot projections which on the one hand make the projections more complicated (i.e. introduce potentially many more crossings) but on the other hand simplify the Seifert pairing obtained from the associated Seifert surface. In Section \ref{section:alexander} I discuss the consequences of the geometric tricks for the Seifert pairing and prove the mail result.

\newpage
\section{Knots projections from trivalent graphs and the $4$-move} \label{section:geometric}
I begin this section with a short summary of Section 3 of \cite{effie}. Let $\mathcal{K}$ be the set of all isotopy classes of oriented knots. In $\mathbb{R}^3$ parameterized by rectangular coordinates $(x,y,z)$ consider the plane $R=\{(x,y,0) : x,y \in \mathbb{R}\}$ and the unit circle in $R$ centered at the origin. For $n \in \mathbb{N}$, an $n$-graph $G$ is a trivalent graph with $2n$ vertices obtained by attaching edges $e_1,\ldots,e_n$ to $S$ as on the left hand side of Fig. \ref{fig:kmap} so that the entire graph lies on $R$ everywhere except when it is necessary to leave $R$ to avoid self intersections. The graph $G$ is weighted by a pair $(w_i,z_i)$ with $w_i \in \mathbb{N}$ and $z_i \in \{ +1,-1\}$. The moves in Fig. 3 of \cite{effie} (which then correspond to knot isotopies) create the set $\mathcal{G}$ of equivalence classes of such graphs. There is a surjective (\cite{suzuki}) map $k: \mathcal{G} \rightarrow \mathcal{K}$ defined as in Fig. \ref{fig:kmap}. 
\begin{figure}[!htb]
	\centering
  \includegraphics[width=.6\textwidth]{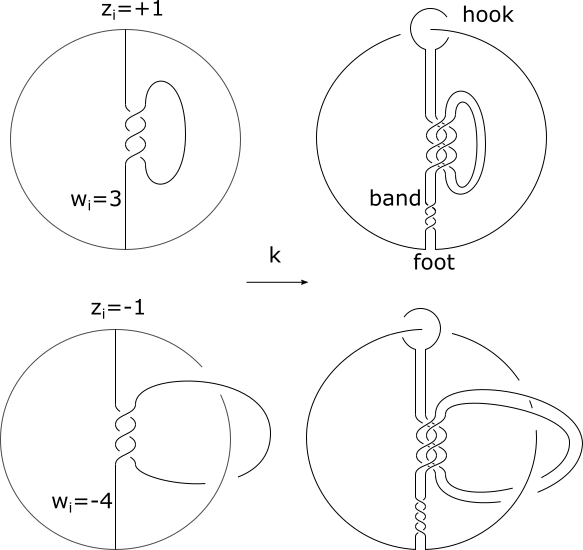}
	\caption{The map $k$ maps $n$-graph projections (left) to knot projections (right) taking into account arc framing and hook sign: a contained graph (top) and a non-contained one (bottom). }
	\label{fig:kmap}
\end{figure}

The degrees of freedom available to us when we create knot projections from $n$-graphs involve the number of edges of the graph, the placement of the attaching feet of these edges on the circle $S$ (the "attaching data"), the entanglement of these edges with themselves, the other edges and the circle $S$ and finally the weights $(w_i, z_i)$ for each edge.

A graph will be called {\bf contained} if there are no crossings involving any of its edges with $S$ in which case the edges can be isotoped in the interior of $S$. The top graph of Fig. \ref{fig:kmap} is contained while the bottom one is not. The following lemma shows that every knot is in the $k$-image of a contained graph (with possibly more edges).

\begin{lemma}\label{lemma:contained}
Every knot has a projection $k(G)$ of a contained graph $G$.
\end{lemma}
\begin{proof} Let $K$ by any knot and let $G$ be a graph such that $k(G)=K$. If all edges (bands)\footnote{We can think either in terms of $G$ isotopies or band isotopies in $K$.} are projected on top of $S$ then we are done. For any band $b$ which under-crosses $S$ proceed as follows. At each under-crossing  of $b$ with $S$ lift the band above $S$ at the expense of creating two new bands whose feet are attached to the band $b$ as in Fig. \ref{fig:lift}. 

\begin{figure}[!htb]
	\centering
  \includegraphics[width=.7\textwidth]{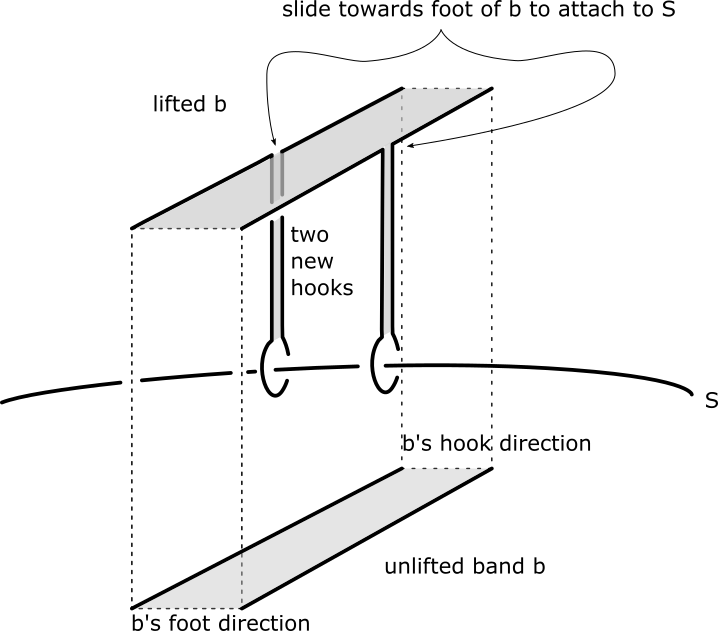}
	\caption{A band's $b$ under-crossing of $S$ becomes an overcrossing at the expense of adding two new bands attached to it which can then be slid onto $S$ along $b$.}
	\label{fig:lift}
\end{figure}

Now slide the feet of these newly introduced bands along the band $b$ towards its foot and onto $S$. If we do this in order of appearance of under-crossings of $b$ as we traverse it from its foot towards its hook then the newly introduced bands will move parallel to the segment of $b$ lying above $S$ and never under-cross $S$. The graph we are seeking is now the pre-image of this new knot projection\footnote{The cost of assuming the Graph is contained is of course that it (and the corresponding knot projection) get more complicated and have more bands/edges.}. 
\end{proof}

We will from now on always assume our graphs are contained. We prove a series of geometric lemmas which we will use to construct a well behaved Seifert surface. The following lemma is self evident and is a straightforward generalization of the kinds of moves I used in \cite{nikos1, nikos2}:

\begin{lemma}\label{lemma:clasp}
Using 4-moves we can swap the clasp of a hook as well as add any number of positive or negative half twists to its band. We will call such a move a {\bf framing move.} A framing move is thus the aggregation of a number of $m$-moves and isotopies. 
\end{lemma}
\begin{proof} The reader can easily verify that the top of Fig. \ref{fig:clasp} is simply a 4-move while the bottom is 4-moves and isotopies. 
\begin{figure}[!htb]
	\centering
  \includegraphics[width=.7\textwidth]{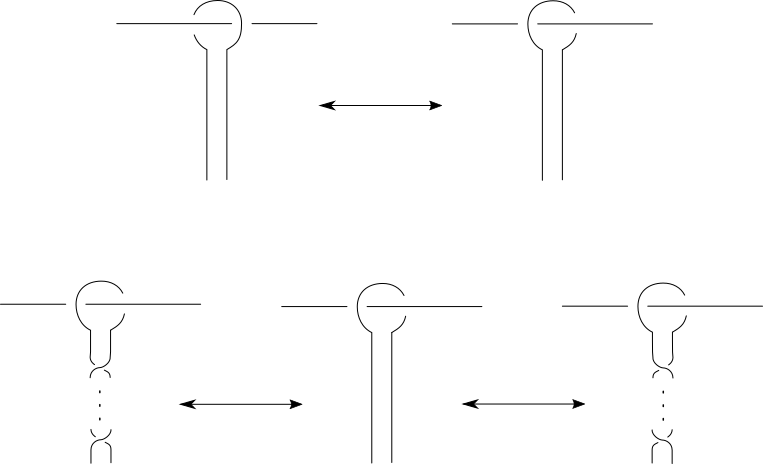}
	\caption{We can change the sign of a clasp by a 4-move and we can add any number of half twists of any sign to the band of a clasp by repeated 4-moves and isotopies.}
	\label{fig:clasp}
\end{figure}
\end{proof}

One way to rephrase this lemma is to say that the weights of the graph G do not affect the $4$-move equivalence class of $k(G)$. The attentive reader, who has Fig. \ref{fig:symplectic}  in mind will have realized at this point that Lemma \ref{lemma:clasp} settles our main theorem for knots constructed from contained 1-graphs. Our second lemma is a wholesale aggregation of $4$-moves (in fact applications of Lemma \ref{lemma:clasp}) and isotopies and will be used to prove two final geometric tricks based on $4$-moves before we move onto the algebraic consequences of our moves:

\begin{lemma}\label{lemma:feet_reduction}
The 4-move induces a move as in Fig. \ref{fig:feet_reduction}:
\begin{figure}[!htb]
	\centering
  \includegraphics[width=.8\textwidth]{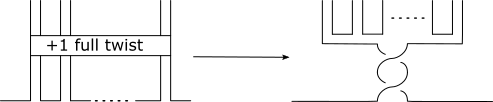}
	\caption{The 4-move applied near the feet of bands}
	\label{fig:feet_reduction}
\end{figure}
\end{lemma}

\begin{proof}
Fig. \ref{fig:feet_example} conveys the flavor of the mechanics of the proof. 
\begin{figure}[!htb]
	\centering
  \includegraphics[width=.6\textwidth]{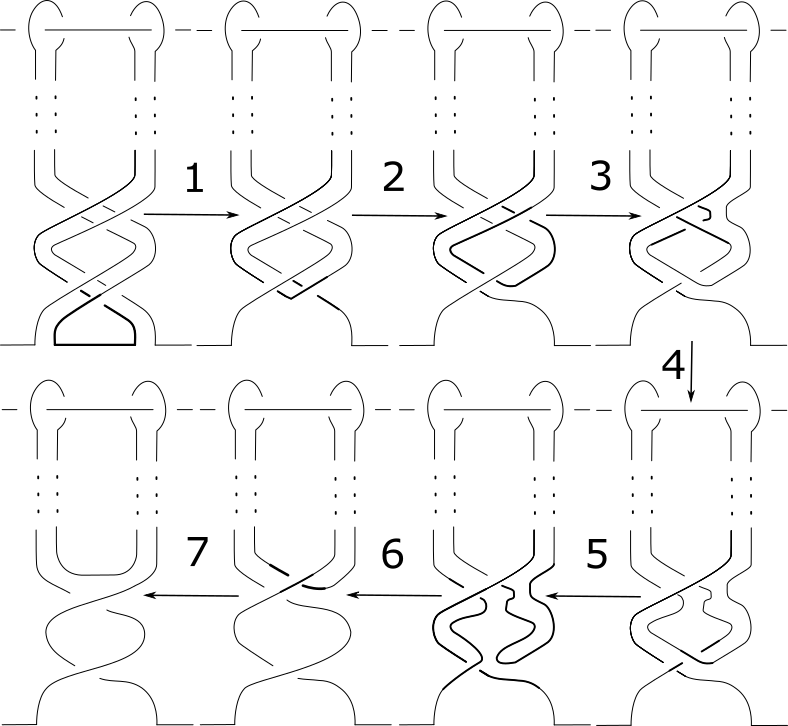}
	\caption{Arrow 1 is simply isotopy. Arrow 2 uses the hook of the right hand side band to get rid of a half twist on it. Arrow 3 is isotopy. Arrows 4 and 5 use the hook of the left hand side band to get rid of two half twists on it. Arrow 6 is isotopy and arrow 7 use the right hand side hook again. The final outcome is one 4-move away from what we would get if we started from a -1 full twist of the bands.}
	\label{fig:feet_example}
\end{figure}
The reader can easily see that this generalizes by simple induction.
\end{proof}

Lemma \ref{lemma:feet_reduction} can also be proven by observing that transferring the full twist from the bands to the circle $S$ as shown in Fig. \ref{fig:feet_example} is simply isotopy. Such an isotopy would then add twists to the bands which we eliminate by using Lemma \ref{lemma:clasp}.  

\noindent Lemmas \ref{lemma:band_twist} and  \ref{lemma:feet} below are straightforward consequences of Lemma \ref{lemma:feet_reduction}

\begin{lemma}\label{lemma:band_twist}
Using 4-moves and isotopies we can add positive (negative) full twists on any number of bands at the expense of introducing a new band each time as in Fig. \ref{fig:band_twist}. We will call such a move a {\bf 2-move on bands}.
\begin{figure}[!htb]
	\centering
  \includegraphics[width=.6\textwidth]{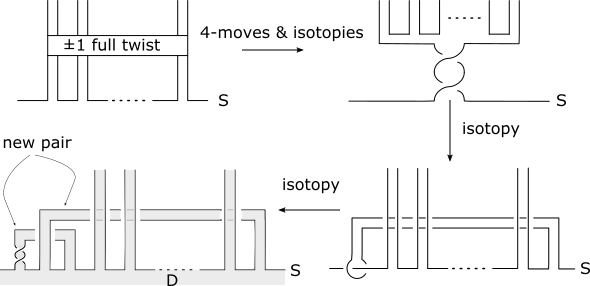}
	\caption{By using 4-moves and isotopies we can add a full twist between any two number of bands by adding a new pair.  }
	\label{fig:band_twist}
\end{figure}
\end{lemma}
\begin{proof} The proof uses Lemma \ref{lemma:feet_reduction} and is easy to see. 
\end{proof}

Notice that in Fig. \ref{fig:band_twist} we provide a preview of the effect on symplectic basis of $H_1(K)$  that we will discuss later. Notice that the sign of the clasp and whether or not the new band is above or below the bands we are treating is obviously a matter of free choice.

\begin{lemma} \label{lemma:feet}
The 4-move induces a 4-move on bands as in Fig. \ref{fig:feet}. Such a move will be referred to us a {\bf 4-move on bands.}
\begin{figure}[!htb]
	\centering
  \includegraphics[width=.7\textwidth]{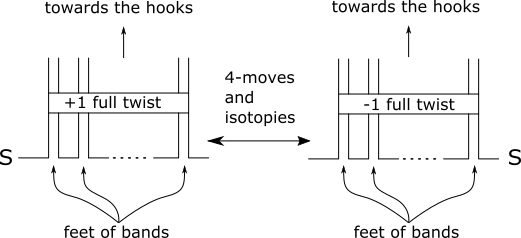}
	\caption{The 4-move induces a 4-move on bands}
	\label{fig:feet}
\end{figure}
\begin{proof} The proof follows by observing that the sign of the full twist on the right hand side of Fig. \ref{fig:feet_reduction} can be changed by one $4$-move. 
\end{proof}
 \end{lemma}
 
Now we start thinking about a certain Seifert surface for our knot projection. Let $K=k(G)$ by any knot (with $G$ contained). Using isotopies as in Fig.\ref{fig:symplectic} 
\begin{figure}[!htb]
	\centering
  \includegraphics[width=.6\textwidth]{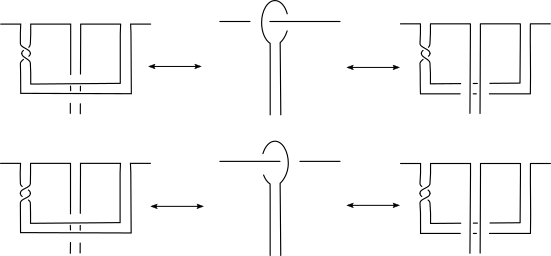}
	\caption{Isotoping the hook of a band so as to create a pair of handles of a Seifert surface. Notice that while the type of the hook determines whether we have a positive or negative full twist on the U-shaped handle we can choose whether the two handles over or under-cross as we please. }
	\label{fig:symplectic}
\end{figure}

we can attach handles to the disk $D$ bounded by $S$ to get a starting Seifert surface $\Sigma$ for our knot as indicated by Fig. \ref{fig:basis} which also shows a symplectic basis $v_i,u_i \in H_1(\Sigma)$ for $i=1,\ldots,n$. 
\begin{figure}[!htb]
	\centering
  \includegraphics[width=.5\textwidth]{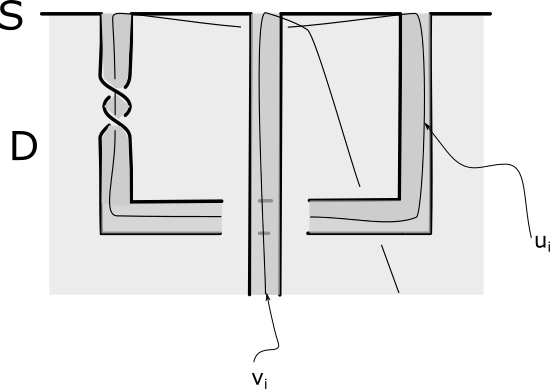}
	\caption{A symplectic basis for $H_1(\Sigma)$}
	\label{fig:basis}
\end{figure}

\section{Alexander polynomial and the $4$-move}\label{section:alexander}

We now discuss the consequences of our geometric observations from Section \ref{section:geometric} for the Seifert pairing in terms of the basis constructed in Fig. \ref{fig:basis}:

\begin{equation} \label{seifert_pairing_map}
\theta \colon H_1(\Sigma) \times H_1(\Sigma) \rightarrow  \mathbb{Z}.
\end{equation}

The {\bf framing move}, {\bf 2-move on bands} and {\bf 4-move on bands} are aggregate 4-moves and isotopies and can be used to change $\theta$ while staying within the same 4-move equivalence class.

The {\bf framing move} of Lemma \ref{lemma:clasp} implies that for $i=1,\ldots,n$ the values $\theta(v_i,v_i)$ can be changed at will after applying 4-moves to our knot.  

The {\bf 4-move on bands} from Lemma \ref{lemma:feet} implies that for $i\ne j$ the values $\theta(v_i,v_j)$ can be changed modulo 2. This is obvious if the bands $b_i,b_j$ of the two basis elements have adjacent feet. If $b_i$ and $b_j$ are non-neighboring let $b$ be the collection of bands standing between them. Now apply the {\bf 4-move on bands} on $b_i,b,b_j$ and once more in reverse on $b_i,b$ to undo the collateral damage from the first one and we are done. 

Any band $b_i$ splits $S$ in two segment $S_l$ and $S_r$ which we might think of as the left- and the right-hand side segments if we use the corresponding arc orientation of $G$ . Any other band $b_j$ either has both ends on the same or on different segments. In the former case we will say that $a_i$ and $a_j$ are {\bf parallel} or {\bf non-interacting} (and write $a_i \parallel a_j$) and in the latter we will say that they are {\bf non-parallel} or {\bf interacting} or {\bf cross} (and write $a_i \perp a_j$). Obviously if $a_i \parallel a_j$ we have $\theta(v_j,v_i)=\theta(v_i,v_j)$ and if $a_i \perp a_j$ then $|\theta(v_j,v_i)-\theta(v_i,v_j)|=1$. 

By applying the {\bf 4-move on bands} from Lemma \ref{lemma:feet} as many times as necessary  we can assume that $\theta(v_j,v_i)=\theta(v_i,v_j)=0$ or $1$ if  $a_i \parallel a_j$. Following that with a single application of the {\bf 2-move on bands} of Lemma \ref{lemma:band_twist} we can assume $\theta(v_j,v_i)=\theta(v_i,v_j)=0$ or $1$ as we please  $a_i \parallel a_j$. In other words $4$-moves change the linkings modulo $2$ and by the {\bf 2-move on bands} we can choose to change the parity as we please at the expense of introducing a new band crossing them.

Similarly by applying the {\bf 4-move on bands} as many times as necessary we can assume and $\theta(v_i,v_j) = 0 \wedge \theta(v_j,v_i)=1$ or  $\theta(v_i,v_j) = -1 \wedge \theta(v_j,v_i)=0$ for $i>j$, if $\alpha_i\perp \alpha_j$. Following that with a single application of the {\bf 2-move on bands} we can then achieve any of the two options.  We can set these values by inducing positive or negative full twists between two bands. If $b_i$ is adjacent to $b_j$ this is obvious else if $b$ stands for the intermediate bands we can apply the moves once more in reverse on $b_i,b$. Notice that the pairs of new bands introduced with each application of Lemma \ref{lemma:band_twist}  already satisfy the conditions they affect on the bands in question as can be read off of Fig. \ref{fig:band_twist}.

Applying {\bf framing moves} we can assume $\theta(v_i,u_i)=0$ and set each $\theta(u_i,v_i)$  and $\theta(u_i,u_i)$ equal to $\pm 1$ as we please. By setting $\alpha_{i,j} = \theta(v_i,v_j)$ we now have:

\begin{lemma} \label{lemma:seifert}
Every knot is 4-move equivalent to a knot with a Seifert matrix of the form:

$$V= \begin{bmatrix}
 A& 0  \\
 \Delta_{\epsilon} & \Delta_{\delta} \\
\end{bmatrix}$$
where $A$ has entries $a_{i,j}$ such that:  
\begin{itemize}
\item the $a_{i,i}$'s can be set at will,
\item ($a_{i,j},a_{j,i}) =(1,0)$, $(0,-1)$, chosen at will (interacting case) or $a_{i,j}=a_{i,j}=0$,$1$, chosen at will (non-interacting case)  for $i \neq j$ and all being changebale module $2$, 
\item $\Delta_{\epsilon}$ (resp. $\Delta_{\delta}$) is a diagonal matrix whose entries $\epsilon_i$ (resp. $\delta_i$) can each be set to $+1$ or $-1$ at will $\forall i$.
\end{itemize}
\end{lemma}

Notice that the framing move just changes the framing of arcs and the sign of their end in $G$ while the 4-move introduces twists between the bands near their feet but preserves the number of arcs of $G$ and their attachment data. The 2-move on bands introduces a full positive or negative twist between two bands but at the expense of introducing a new edge in $G$. They all manipulate the algebra of the Seifert matrix but produce a much more complicated knot projection than the one we started out with.

Our main theorem follows from the following lemma:
\begin{lemma}
Let $$V= \begin{bmatrix}
 A& 0  \\
 \Delta_{\epsilon} & \Delta_{\delta} \\
\end{bmatrix}$$ 

\noindent be an $n\times n$ which is $2\times2 $ block matrix as in Lemma \ref{lemma:seifert} with $a_{i,j}\equiv 0 \mod{2}$  for $i \neq j$ in the non-interacting case. Then there are values $a_{i,j}$ such that: $$|\det(V-tV^\top)| = t^n$$.
\end{lemma}

\begin{proof}
By Theorem 3 of \cite{block}, on the determinant of $2\times 2$ block matrices, we have $\det(V-tV^\top) = \det([(\alpha_{i,j}-\alpha_{j,i}t)(1-t)\delta_{j}+\delta_{i,j}t]_{n\times n})$. Here $\delta_{i,j}$ is the Kronecker delta. By multiplying each column by $\delta_j$, since $\delta_j^2=1$, this is then up to sign equal to $\det([(\alpha_{i,j}-\alpha_{j,i}t)(1-t)+\delta_{j}\delta_{i,j}t]_{n\times n})$. We will now proceed by induction using elementary row and column operations. 

For $n=1$ our determinant is simply $\alpha_{1,1}(1-t)^2+\delta_it$. So $\alpha_{1,1}=0 \wedge \delta_i=\pm1$ does it which we achieve by {\bf framing moves}.

Working out the case $n=2$ will teach us the mechanics of the process. Our matrix now looks as follows:
$$
\begin{bmatrix}
  \alpha_{1,1}(1-t)^2+\delta_1t& (\alpha_{1,2}-\alpha_{2,1}t) (1-t) \\
 (\alpha_{2,1}-\alpha_{1,2}t) (1-t) & \alpha_{2,2}(1-t)^2+\delta_2 t\\
\end{bmatrix}
$$
We distinguish two cases: either $b_1\parallel b_2$ or $b_1\perp b_2$ and examine them separately. 

\begin{case}
         Suppose that $b_1\parallel b_2$.
\end{case}

 Then $\alpha_{1,2} = \alpha_{2,1}$ and by the {\bf 4-move on bands} we can change them at will modulo 2. 

Since $\alpha_{1,2} = \alpha_{2,1} \equiv 0 \mod 2$ we set $\alpha_{1,1} = \alpha_{2,2} = 0$ and $\delta_1 = \delta_2 =\pm 1 $ using {\bf framing moves} and we are done.

\begin{case}
         Suppose that $b_1\perp b_2$.
\end{case} 

Using {\bf 4-moves on bands} we can set $\alpha_{1,2} =0 \wedge \alpha_{2,1} =1$ or $\alpha_{1,2} =-1 \wedge \alpha_{2,1} =0$. Now our matrix looks like:

$$
\begin{bmatrix}
  \alpha_{1,1}(1-t)^2+\delta_1t& -t (1-t) \\
  (1-t) & \alpha_{2,2}(1-t)^2+\delta_2 t\\
\end{bmatrix}
$$

We now finish by setting $\alpha_{i,i}=i-1\wedge \delta_i=-1$ for $i=1,2$ and perform obvious determinant preserving elementary operation ($(1-t)$-shears) to make the matrix lower diagonal with $\pm t$'s on the diagonal. 
Now assume it works for $n-1$ and lets work with a knot with $n$ handles (contained and so that parallel bands link evenly).

From our discussion so far we can choose a $4$-move equivalence class representative such that our matrix $V-tV^\top$ looks like:

\begin{equation} \label{induction}
V-tV^\top =
\left[
\begin{array}{c|cc}
t & -t(1-t)\mathbf{1}^T & \mathbf{0}^T \\ \hline
(1-t)\mathbf{1} & M_{11} & M_{12} \\
\mathbf{0} & M_{21} & M_{22}
\end{array}
\right]
\end{equation}
where $\mathbf{1}$ is a column vector of 1’s of appropriate size and $\mathbf{0}$ is zeros. To see that notice that the first band in our band enumeration which corresponds to our basis for $H_1(K)$ is perpendicular to some and parallel to other bands. Using $4$-moves, for bands $i$ to which it is perpendicular set $\alpha_{1i}=0$ and $\alpha_{i1}=1$ and for bands $j$ to which it is parallel set $\alpha_{1j}=\alpha_{j1}=0$.

If band one is parallel to all others then the first row of our matrix in equation \ref{induction} contains only $0$'s off diagonal and induction applies. If there are bands perpendicular to band one then we can apply $(1-t)$-shears by adding $(1-t)C_1$ to those columns $C_i$ of  $V-tV^\top$  which start with $-t(1-t)$. Notice that, in doing so, we are adding $(1-t)^2$ to $M{11}$ so in order for induction to apply we use $4$-moves to set all $\alpha_{ij}^{'}$ involved in $M{11}$ to $\alpha_{ij}-1$ where $\alpha_{ij}$'s are the ones our induction hypothesis asserts and can be set by $4$-moves.
\end{proof}

\section{Discussion}

As in the unknotting of the $2$--cable of the trefoil in \cite{nikos1, nikos2}, where it seemed necessary to pass through diagrams of higher complexity, the present work likewise requires introducing substantial geometric complexity via $4$--moves in order to locate a representative of the $4$--move class with trivial Alexander polynomial.

Define the \emph{$4$--move height} of a knot as follows. For every $4$--move unknotting movie, consider the maximum number of crossings appearing in the movie; take the minimum of these maxima over all movies, and subtract the crossing number of the original knot. Even if Nakanishi's conjecture is correct, this height may grow rapidly as a function of classical invariants of the knot (crossing number, unknotting number, etc.), reminiscent of the behavior of Dehn's function in geometric group theory, making it hard to prove. One may likewise define a \emph{$4$--move Alexander--polynomial height}.

The constructions here do not yield any upper bound on the Alexander--polynomial height, even for knots with fixed classical unknotting number. For $n=1$, for example, our map $k$ essentially produces Whitehead doubles. The subsequent $4$--move simplification does produce a representative with trivial Alexander polynomial, but only by increasing the crossing number by four times the writhe of the underlying Whitehead--double diagram. In my view, this only deepens the mystery surrounding the $4$--move puzzle.

In our $n$--graph model, the edges $e_i$ have their feet on the base circle $S$, and crossing changes among different $e_i$'s (more precisely, among their images in $k(G)$) are simply $C_3$--moves (clasp-pass moves). Self crossings of the $e_i$'s are induced by a pair of $C_2$--moves. One might try to find cases where $4$--moves induce $C_2$ or $C_3$ moves. I can establish that they induce $C_3$ moves only conditionally and only in local configurations near the feet of the bands. This suffices to trivialize the Alexander polynomial, but I do not see how to extract anything stronger. Also, at present, I can generate $C_2$--moves only in very special configurations, one of which appears in the unknotting movies of \cite{nikos1, nikos2}. 

\bibliography{simple}

\end{document}